\documentclass[12pt,a4paper,fleqn]{scrartcl}

\RequirePackage{amsthm,amsmath,natbib}
\RequirePackage{amssymb,amstext}
\RequirePackage{hypernat}
\usepackage{graphicx}
\usepackage{enumerate}
\usepackage{bbm}
\usepackage{url}
\usepackage{verbatim}
\usepackage{ bm}
\usepackage{setspace}

\DeclareOldFontCommand{\it}{\normalfont\itshape}{\mathit}

\usepackage[pdfpagemode=UseOutlines ,plainpages=false
,hypertexnames=false ,pdfpagelabels ,hyperindex=true,colorlinks=true]{hyperref}
\makeatletter%
	\Hy@breaklinkstrue%
	\makeatother%
\usepackage{color}
\definecolor{darkred}{rgb}{0.6,0.0,0.1}
\definecolor{darkgreen}{rgb}{0,0.5,0}
\definecolor{darkblue}{rgb}{0,0,0.5}
\hypersetup{colorlinks ,linkcolor=darkblue ,filecolor=darkgreen
,urlcolor=darkblue ,citecolor=darkblue}

\RequirePackage{natbib}
\numberwithin{equation}{section}
\newtheorem{theorem}{Theorem}[section]

\newtheorem{A}{Assumption}
\newtheorem{lemma}{Lemma}[section]

\theoremstyle{definition}

\DeclareMathOperator{\E}{\mathbb{E}}

\def\Var{\mathop{{\mathbb V}ar}\nolimits}%

\newcommand{\R}{\mathbb{R}}

\renewcommand{\Pr}{\mathbb{P}}

\newcommand{\set}[1]{{\left\lbrace #1\right\rbrace }}

\def\1{\mathop{\mathbbm 1}\nolimits}

\renewcommand{\cite}{\citet}
\def\Cov{\mathop{ {\mathbb C}ov}\nolimits}%

\begin{document}
\title{Simple Adaptive Estimation of Quadratic Functionals in Nonparametric IV Models\thanks{We are grateful to Volodia Spokoiny for his friendship and insightful comments, and admire his creativity and high standard on research. We thank Enno Mammen and an anonymous referee for helpful comments, and Cristina Butucea, Tim Christensen, Richard Nickl and Sasha Tsybakov for initial discussions. Early versions have been presented at the conference ``Celebrating Whitney Newey's Contributions to Econometrics'' at MIT in May 2019, the North American Summer Meeting of the Econometric Society in Seattle in June 2019, and the conference ``Foundations of Modern Statistics'' in occasion of Volodia Spokoiny's 60th birthday at WIAS in November 2019.}}

 \author{{\textsc{  Christoph Breunig}}\thanks{Department of Economics, Emory University, Rich Memorial Building, Atlanta, GA 30322, USA. Email:
\url{christoph.breunig@emory.edu}}
  \and{\textsc{Xiaohong Chen}}\thanks{Cowles Foundation for Research in Economics, Yale University, Box 208281, New Haven, CT 06520, USA. E-mail: \url{xiaohong.chen@yale.edu}}}
\date{First version: August 2018; Revised version: January 2022}
\maketitle
\begin{abstract}
{\small This paper considers adaptive, minimax estimation of a quadratic functional in a nonparametric instrumental variables (NPIV) model, which is an important problem in optimal estimation of a nonlinear functional of an ill-posed inverse regression with an unknown operator. We first show that a leave-one-out, sieve NPIV estimator of the quadratic functional
can attain a convergence rate that coincides with the lower bound previously derived in \cite{ChenChristensen2017}. The minimax rate is achieved by the optimal choice of the sieve dimension (a key tuning parameter) that depends on the smoothness of the NPIV function and the degree of ill-posedness, both are unknown in practice. We next propose a Lepski-type data-driven choice of the key sieve dimension adaptive to the unknown NPIV model features. The adaptive estimator of the quadratic functional is shown to attain the minimax optimal rate in the severely ill-posed case and in the regular mildly ill-posed case, but up to a multiplicative $\sqrt{\log n}$ factor in the irregular mildly ill-posed case.}
\end{abstract}


\onehalfspacing
\section{Introduction}

Long before the recent popularity of instrumental variables in modern machine learning causal inference, reinforcement learning and biostatistics, the instrumental variables technique has been widely used in economics. For instance, instrumental variables regressions are frequently used to account for omitted variables, mis-measured regressors, endogeneity in simultaneous equations and other complex situations in economic observational data.
In economics and other social sciences, as well as in medical research, it is very difficult to estimate causal effects when treatment assignment is not randomized. Instrumental variables are commonly used to provide exogenous variation that is associated with the treatment status, but not with the outcome variable (beyond its direct effect on the treatments).

To avoid mis-specification of parametric functional forms, nonparametric instrumental variables (NPIV) regressions have gained popularity in econometrics and modern causal inference in statistics and machine learning. The simplest NPIV model assumes that a random sample $\{(Y_i,X_i,W_i)\}_{i=1}^n$ is drawn from an unknown joint distribution of $(Y,X,W)$ satisfying
\begin{align}\label{npiv:model}
Y = h_0(X) + U,~~~\E[U|W] = 0,
\end{align}
where $h_0$ is an unknown continuous function, $X$ is a $d$-dimensional vector of continuous endogenous regressors in the sense that $\E[U|X]\neq 0$, $W$ is a vector of conditioning variables (instrumental variables) such that $\E[U|W] = 0$.
The structural function $h_0$ can be identified as a solution to an integral equation of first kind with an unknown operator:
\begin{align*}
\E[Y|W=w]=(Th_0)(w):=\int h_0(x)f_{X|W}(x|w)dx,
\end{align*}
where the conditional density $f_{X|W}$ (and hence the conditional expectation operator $T$) is unknown. Under mild conditions, the conditional density $f_{X|W}$ is continuous and the operator $T$ smoothes out ``low regular'' (or wiggly) parts of $h_0$. This makes the nonparametric estimation (recovery) of $h_0$ a difficult ill-posed inverse problem with an unknown smoothing operator $T$. See, for example, \cite{NP03}, \cite{HH}, \cite*{CFR06handbook}, \cite*{BCK07econometrica}, \cite{ChenReiss2011} and \cite{DFFR}. For a given smoothness of $h_0$, the difficulty of recovering $h_0$ depends on the smoothing property of the conditional expectation operator $T$. The literature distinguishes between the mildly and severely ill-posed regimes, and the optimal convergence rates for nonparametrically estimating $h_0$ are different in the two regimes.

This paper considers adaptive, minimax rate-optimal estimation of a quadratic functional of $h_0$ in the NPIV model (\ref{npiv:model}):
\begin{align}\label{Qf}
f(h_0):=\int h_0^2(x) \mu(x) dx
\end{align}
for a known positive, continuous weighting function $\mu$, which is assumed to be uniformly bounded below from zero and from above on some subset of of the support of $X$. Let $\widehat{h}$ be a sieve NPIV estimator of the NPIV function $h_0$ (see e.g., \cite{BCK07econometrica}). \cite{chen2013} and \cite{ChenChristensen2017} considered inference on a slightly more general nonlinear functional $g(h_0)$ using plug-in sieve NPIV estimator $g(\widehat{h})$. However, there is no result on any adaptive, minimax rate-optimal estimation of any nonlinear functional $g(h_0)$ of the NPIV function $h_0$ yet. Since a quadratic functional is a leading example of a smooth nonlinear functional in $h_0$, \cite[Theorem C.1]{ChenChristensen2017} established the minimax lower bound for estimating a quadratic functional $f(h_0)$ in a NPIV model. They also point out that a plug-in sieve NPIV estimator $f(\widehat{h})$ of the quadratic functional $f(h_0)$ can achieve the lower bound in the severely ill-posed regime, but fails to achieve the lower bound in the mildly ill-posed regime. Moreover, none of the existing work considers adaptive minimax rate-optimal estimation of the quadratic functional $f(h_0)$ in a NPIV model.

In this paper, we first propose a simple leave-one-out sieve NPIV estimator $\widehat{f}_J$ for the quadratic functional $f(h_0)$,
and establish an upper bound on its convergence rate.
 By choosing the sieve dimension $J$ optimally to balance the squared bias and the variance parts, we show that the resulting convergence rate of $\widehat{f}_J - f(h_0)$ coincides with the lower bound of \cite[Theorem C.1]{ChenChristensen2017}. In this sense the estimator $\widehat{f}_J$ is minimax rate-optimal for $f(h_0)$ regardless whether the NPIV model is severely ill-posed or mildly ill-posed.
In particular, for the severely ill-posed case, the optimal convergence rate
is of the order $(\log n)^{-\alpha}$, where $\alpha>0$ depends on the smoothness of the NPIV function $h_0$ and the degree of severe ill-posedness. For the mildly ill-posed case, the optimal convergence rate of $\widehat{f}_J - f(h_0)$ exhibits the so-called \emph{elbow phenomena}: the rate is of the parametric order $n^{-1/2}$ for the regular mildly ill-posed case, and is of the order $n^{-\beta}$ for the irregular mildly ill-posed case, where $\beta\in (0,1/2)$ depends on the smoothness of $h_0$, the dimension of $X$ and the degree of mild ill-posedness.

The minimax optimal estimation rate of $\widehat{f}_J - f(h_0)$ is achieved by the optimal choice of the sieve dimension $J$ (a key tuning parameter) that depends on the unknown smoothness of $h_0$ and the unknown degree of ill-posedness. We next propose a data driven choice $\widehat{J}$ of the sieve dimension based on a modified Lepski method.\footnote{See \cite{Lepski90}, \cite{lepski1997} and \cite*{lepski1997optimal} for detailed descriptions of the original Lepski principle.} The modification is needed to account for the estimation of the unknown degree of ill-posedness. The adaptive, leave-one-out sieve NPIV estimator $\widehat{f}_{\widehat{J}}$ of $f(h_0)$ is shown to attain the minimax optimal rate in the severely ill-posed case and in the regular mildly ill-posed case, but up to a multiplicative $\sqrt{\log n}$ in the irregular mildly ill-posed case. We note that even for adaptive estimation of a quadratic functional of a direct regression in a Gaussian white noise model, \cite{efromovich1996} already shown that the extra $\sqrt{\log n}$ factor is the necessary price to pay for adaptation to the unknown smoothness of the regression function.

Previously for the nonparametric estimation of $h_0$ in the NPIV model (\ref{npiv:model}), \cite{horowitz2014adaptive} considers adaptive estimation of $h_0$ in $L^2$ norm using a model selection procedure. \cite{breunig2016} consider adaptive estimation of a linear functional of the NPIV function $h_0$ in a root-mean squared error metric using a combined model selection and Lepski method. These papers obtain adaptive rate of convergence up to a multiplicative factor of $\sqrt{\log(n)}$ (of the minimax optimal rate) in both severely ill-posed and mildly ill-posed cases. \cite*{chen2021} propose adaptive estimation of $h_0$ in $L^\infty$ norm using a modified Lepski method and tight random matrix inequalities to account for the estimated measure of ill-posedness. They show that their data-driven procedure attains the minimax optimal rate in $L^\infty$ norm and is fully adaptive to the unknown smoothness of $h_0$ in both severely and mildly ill-posed regimes. Our data-driven choice of the sieve dimension is closest to that of \cite{chen2021}, which might explain why we also obtain minimax optimal adaptivity for the quadratic functional $f(h_0)$ in both severely and mildly ill-posed regimes.

While \cite{horowitz2014adaptive}, \cite{breunig2016} and \cite{chen2021} use plug-in sieve NPIV estimators in their adaptive estimation of a linear functional of $h_0$, we use a leave-one-out sieve NPIV estimator $\widehat{f}_J$ for the quadratic functional $f(h_0)=\int h_0^2(x) \mu(x) dx$. Recently \cite{BC2020} propose a test statistic that is based on a standardized leave-one-out estimator of a quadratic distance for a null hypothesis of $\E[(h_0(X)-h^R(X))^2\mu(X)]=0$ in a NPIV model (for some parametric, semiparametric or shape restricted $h^R$). They construct an adaptive minimax test using a random exponential scan procedure. We use the unstandardized leave-one-out estimator $\widehat{f}_J$ in our modified Lepski procedure for adaptive minimax estimation of $f(h_0)$ in a NPIV model.
It is well-known that adaptive minimax testing and adaptive minimax estimation are related but different (see, e.g., \cite{nicklbook}).
In particular, while both papers apply a tight Bernstein-type inequality
for U-statistics (\cite*{houdre2003}) in the proofs, the adaptive optimal rates are different. For instance, the adaptive minimax $L^2$ separation rate of testing in \cite{BC2020} is always slower than $n^{-1/2}$, while our adaptive minimax estimation for $f(h_0)$ can achieve the parametric rate of $n^{-1/2}$ for regular mildly ill-posed NPIV models.

Minimax rate-optimal estimation of a quadratic functional in density and direct regression (in Gaussian white noise) settings has a long history in statistics. See, for example, \cite{bickel1988estimating}, \cite{donoho1990minimax}, \cite{Fan91},
\cite{efromovich1996}, \cite{laurent2000}, \cite{CL06}, \cite{gine2008}, \cite*{collier2017} and the references therein.
To the best of our knowledge, there are not many published papers on minimax estimation of a quadratic functional in difficult inverse problems. See \cite{butucea2007}, \cite{butucea2011}, \cite{che201} and \cite{kroll2019rate} for deconvolutions and inverse regressions in Gaussian sequence models. Moreover, \cite{che201} seems the only published work on adaptive estimation of a quadratic functional in a special deconvolution (with a known operator). Our paper is the first to propose a simple estimator that is adaptive minimax rate-optimal for a quadratic functional in a NPIV model, and also contributes to inverse problems with unknown operators.

The rest of the paper is organized as follows. Section \ref{sec:minimax} presents the leave-one-out sieve NPIV estimator of the quadratic functional $f(h_0)$, and derives its optimal convergence rates. Section \ref{sec:adapt} first presents a simple data-driven procedure of choosing the sieve dimension using a modified Lepski method. It then establishes the optimal convergence rates of our adaptive estimator of the quadratic functional. Section \ref{sec:end} provides a brief conclusion and discusses several extensions. All proofs can be found in the Appendices \ref{appendix:minimax}--\ref{appendix:lemmas}.

\section{Minimax Optimal Quadratic Functional Estimation}\label{sec:minimax}

This section consists of three parts. The first subsection introduces model preliminaries and notation. Subsection \ref{subsec:est} introduces a simple leave-one-out, sieve NPIV estimator of the quadratic functional $f(h_0)$. Subsection \ref{subsec:rate}
establishes the convergence rate of the proposed estimator, and shows that the convergence rate coincides with the lower bound and hence is optimal.

\subsection{Preliminaries and Notation}\label{subsec:notation}

We first introduce notation that is used throughout the paper. For any random vector $V$ with support $\mathcal V$, we let $L^2(V)=\{\phi:\mathcal V \to \R, \|\phi\|_{L^2(V)}<\infty\}$ with the norm $\|\phi\|_{L^2(V)}=\sqrt{\E[\phi^2(V)]}$. If $\{a_n\}$ and $\{b_n\}$ are sequences of positive numbers, we use the notation $a_n \lesssim b_n$ if $\limsup_{n\to\infty}a_n/b_n<\infty$ and $a_n\sim b_n$ if $a_n\lesssim b_n$ and $b_n\lesssim a_n$.

We consider a known positive, continuous weighting function $\mu$, which is assumed to be uniformly bounded below from zero and from above on some subset of $\mathcal X$, denoted by $X_\mu$. Denote $L_\mu^2 =\{h:\mathcal X_\mu \to \R,\|h\|_\mu<\infty\}$ with the norm $\|h\|_\mu =\sqrt{\int h^2(x)\mu(x)dx}$.
We consider basis functions $\{\psi_j\}_{j\geq 1}$ to approximate the NPIV function $h_0$.  Its orthonormalized analog with respect to $\|\cdot\|_\mu$ is denoted by  $\{\widetilde \psi_j\}_{j\geq 1}$.
We assume that the structural function $h_0$ belongs to the Sobolev ellipsoid
\begin{align*}
  \mathcal H_2(p, L)=\Big\{h\in L_\mu^2:\, \sum_{j=1}^\infty j^{2p/d}\langle h, \widetilde \psi_j\rangle_\mu^2\leq L\Big\},~~~\text{for}~~d/2<p<\infty,~~0<L<\infty~.
\end{align*}

Let $T : L^2(X) \mapsto L^2(W)$ denote the conditional expectation operator given by $(T h)(w) = \E[h(X)|W = w]$. Finally let $\set{\psi_1,...,\psi_J}$ and $\set{b_1,...,b_K}$ be collections of sieve basis functions of dimension $J$ and $K$ for approximating functions in $L^2(X)$ and $L^2(W)$, respectively. We define the \textit{sieve measure of ill-posedness} which, roughly speaking, measures how much the conditional expectation operator $T$ smoothes out $h$.
Following \cite{BCK07econometrica} the sieve $L_\mu^2$ measure of ill-posedness is
\begin{align*}
\tau_J = \sup_{h\in\Psi_J, h\neq 0} \frac{\|h\|_{\mu}}{\|Th\|_{L^2(W)}}=\sup_{h\in\Psi_J, h\neq 0} \frac{\sqrt{f(h)}}{\|Th\|_{L^2(W)}},
\end{align*}
where $\Psi_J = \text{clsp}\{\psi_1,...,\psi_J\} \subset L^2(X)$ denotes the sieve spaces for the endogenous variables.
We call a NPIV model \eqref{npiv:model}
\begin{itemize}
\item[(i)] mildly ill-posed if $\tau_j\sim  j^{a/d}$ for some $a > 0$; and
\item[(ii)]  severely ill-posed if $\tau_j\sim \exp(\frac{1}{2} j^{a/d})$ for some $a > 0$.
\end{itemize}

\subsection{A Leave-one-out, Sieve NPIV Estimator}\label{subsec:est}

Let $\{(Y_i,X_i,W_i)\}_{i=1}^n$ denote a random sample from the NPIV model \eqref{npiv:model}. The sieve NPIV (or series 2SLS) estimator $\widehat h$ of $h_0$ can be written in matrix form as follows (see, e.g., \cite{ChenChristensen2017})
\[
 \widehat h(\cdot) = \psi^J(\cdot)'[\Psi'P_B\Psi]^- \Psi'P_B{\textbf{Y}}
                   = \psi^J(\cdot)'\widehat A B'{\textbf{Y}}/n
\]
where $P_B= B(B'B)^-B'$ and ${\textbf{Y}} = (Y_1,\ldots,Y_n)'$,
\begin{align*}
\psi^J(x) = (\psi_1(x),\dots,\psi_J(x))'\qquad  \Psi = (\psi^J(X_1),\dots,\psi^J(X_n))'\\
b^K(w) = (b_1(w),\dots,b_K(w))' \qquad B = (b^K(W_1),\dots,b^K(W_n))'
\end{align*}
and $\widehat A=n[\Psi'P_B\Psi]^- \Psi' B(B'B)^-$ is an estimator of $A=[S'G_b^{-1}S]^{-1}S'G_b^{-1}$, with $S=\E[b^K(W_i)\psi^J(X_i)']$ and $G_b=\E[b^K(W_i)b^K(W_i)']$.

As pointed out by \cite{ChenChristensen2017}, although one could estimate $f(h_0)$ by the plug-in sieve NPIV estimator $f(\widehat h)$, it fails to achieve the minimax lower bound. We propose a leave-one-out sieve NPIV estimator for the quadratic functional $f(h_0)$ as follows:
\begin{align*}
\widehat{f_J}&
=\frac{2}{n(n-1)}\sum_{1\leq i< i'\leq n}Y_i b^K(W_i)'\widehat A'G_\mu\widehat A\,b^K(W_{i'})Y_{i'}
\end{align*}
where $G_\mu=\int \psi^J(x)\psi^J(x)' \mu(x)dx$. We will show that this simple leave-one-out estimator $\widehat{f_J}$ can achieve the lower bound for estimating $f(h_0)$.

Based on many simulation results in \cite{BCK07econometrica} and \cite{ChenChristensen2017}, the crucial \emph{regularization} parameter in sieve NPIV estimation of $h_0$ is the dimension $J$ of the sieve space used to approximate unknown function $h_0$. In this paper, we simply let $K(J)=c_KJ$ for some constant $c_K\geq 1$.
Further, we let $\zeta_{\psi,J}=\sup_x\|G_\mu ^{-1/2}\psi^J(x)\|$ and $\zeta_{b,K}=\sup_w\|G_b^{-1/2}b^K(w)\|$. For instance, $\zeta_{\psi,J} = O(\sqrt  J)$ and $\zeta_{b,K} = O( \sqrt K)$ for (tensor-product) polynomial spline, wavelet and cosine bases. Denote $\zeta_J=\max(\zeta_{\psi,J}, \zeta_{b,K})$ for $K=K(J)$.
In the rest of the paper we restrict sieve bases to the ones such that $\zeta_J = O(\sqrt  J)$.

\subsection{Rate of Convergence}\label{subsec:rate}

We first introduce assumptions that are used to derive our rate of convergence of the estimator $\widehat{f_J}$. We denote the sieve Least Squares (LS) projection of $h$ onto $\Psi_J=\text{clsp} \{\psi_1,...,\psi_J\}$ as $\Pi_J h(x)=\psi^J(x)'G_\mu^{-1} \langle \psi^J, h\rangle_\mu$. For $h_0\in \mathcal H_2(p, L)$ we have $\|h_0-\Pi_J h_0\|_{\mu} \leq L J^{-p/d}$ which is used throughout this paper. This implies that $\sqrt{J(\log J)} \|h_0-\Pi_Jh_0\|_\mu=o(1)$ as $J$ goes to infinity (since $p>d/2$).

\begin{A}\label{A:ident}
(i) $T[h-h_0]=0$ for any $h\in L_\mu^2$ implies that $f(h)=f(h_0)$;
(ii) $\sup_{w\in\mathcal W}\E[Y^2|W=w]\leq \overline\sigma_Y^2<\infty$ and $\E[Y^4]<\infty$;
(iii)
the densities of $X$ and $W$ are Lebesgue continuous and uniformly bounded below from zero and from above on the closed rectangular supports $\mathcal X \subset \mathbb R^d$ and $\mathcal W \subset \mathbb R^{d_w}$, respectively.
\end{A}

\begin{A}\label{A:basis}
$\tau_J J\sqrt{(\log J)/ n}=O(1)$.
 \end{A}

Below we let $\Pi_K g (w)=b^K(w)'G_b^{-1} \E[b^K (W) g(W)]$ denote the sieve LS projection of $g\in L^2 (W)$ onto $B_K = \text{clsp} \{b_1,...,b_K\}$.

\begin{A}\label{A:link}
(i) $\sup_{h\in\Psi_J} \tau_J \|(\Pi_K T-T)h\|_{L^2(W)} /\|h\|_\mu\leq v_J$ where $v_J<1$ for all $J$ and $v_J\to 0$ as $J\to\infty$.
  (ii) there exists a constant $C>0$ such that $\tau_J \|T(h_0-\Pi_J h_0 )\|_{L^2(W)}\leq C \|h_0-\Pi_J h_0 \|_\mu$.
\end{A}

For a $r\times c$ matrix $M$ with $r \leq  c$ and full row rank $r$ we let $M_l^-$ denote its left pseudoinverse, namely $(M'M)^-M'$ where $'$ denotes transpose and $^-$ denotes generalized inverse. Below, $\|\cdot\|$ respectively denotes the vector $\ell_2$ norm when applied to a vector and the operator norm $\|A\|:=\sup_{x:\|x\|=1}\|Ax\|$ when applied to a matrix $A$. Let $(s_1,\dots,s_J)$ denote the singular values, in non-increasing order, of $G_b^{-1/2}SG_\mu^{-1/2}$. In particular
$s_J=s_{\min}(G_b^{-1/2}S G_\mu^{-1/2})$.

\begin{A}\label{A:per}
$\big\|\textsl{diag}(s_1,\dots,s_J)\big(G_b^{-1/2}SG_\mu^{-1/2}\big)_l^-\big\|\leq D$ for some constant $D>0$.
\end{A}

\textit{Discussion of Assumptions:}
Assumption \ref{A:ident}(i) ensures identification of the nonlinear functional $f(h_0)$.
Assumption \ref{A:basis} restricts the growth of the sieve dimension $J$.
Assumption \ref{A:link}(i) is a mild condition on the approximation properties of the basis used for the instrument space and is first imposed in \cite{chen2021}. In fact, $\|(\Pi_K T-T)h\|_{L^2(W)}= 0$ for all $h \in \Psi_J$ when the basis functions for $B_K$ (with $K\geq J$) and $\Psi_J$ form either a Riesz basis or an eigenfunction basis for the conditional expectation operator. Assumption \ref{A:link}(ii) is the usual $L^2$ “stability condition” imposed in the NPIV literature (cf. Assumption 6 in \cite{BCK07econometrica}). Note that Assumption \ref{A:link}(ii) is  also automatically satisfied by Riesz bases.
Assumption \ref{A:per} is a modification of the sieve measure of ill-posedness and was used by \cite{EfromovichKoltchinskii2001}. Assumption \ref{A:per} is also related to the extended link condition in  \cite{breunig2016} to establish optimal upper bounds in the context of minimax optimal estimation of linear functionals in NPIV models.
Finally we note that by definition, $s_J$ satisfies
\begin{align}\label{s_J:lower}
s_J =\inf_{h\in\Psi_J, h\neq 0}\frac{\|\Pi_K T h \|_{L^2(W)}}{\|h\|_\mu}\leq \tau_J^{-1}
\end{align}
for all $K=K(J)\geq J>0$. Assumption \ref{A:link}(i) further implies that
\begin{align}
s_J \geq \inf_{h\in\Psi_J, h\neq 0}\frac{\|T h \|_{L^2(W)}}{\|h\|_\mu}-\sup_{h\in\Psi_J, h\neq 0}\frac{\|(\Pi_K T -T)h \|_{L^2(W)}}{\|h\|_\mu} =c_\tau\tau_J^{-1},\label{s_J:upper1}
\end{align}
for some constant $c_\tau>0$. We shall maintain Assumption \ref{A:link}(i) and use the equivalence of $s_J$ and $\tau_J^{-1}$ in the paper.

The next result provides an upper bound on the rate of convergence for the estimator $\widehat{f}_J$.

\begin{theorem}\label{thm:rate:quad:fctl}
Let Assumptions \ref{A:ident}--\ref{A:link} hold. Then:
\begin{equation}\label{rate:extend:cond}
\widehat f_J-f(h_0)=O_p\left(\frac{\tau_J^2 \sqrt J}{ n} +\frac{\big\|\langle h_0,\psi^J\rangle_\mu'(G_b^{-1/2}S)_l^-\big\| +\tau_J\|h_0-\Pi_J h_0\|_\mu}{\sqrt{n}}- \|h_0-\Pi_J h_0\|_\mu^2\right).
\end{equation}
If in addition $h_0\in \mathcal H_2(p, L)$ and Assumption \ref{A:per} holds, then:
\begin{enumerate}
\item Mildly ill-posed case: choosing $J\sim n^{2d/(4(p+a)+d)}$ implies
\begin{align}\label{rate:mild}
\widehat f_J-f(h_0)=
\left\{\begin{array}{lcl}
O_p\left(n^{-4p/(4(p+a)+d)}\right), && \mbox{if $p\leq a+d/4$},\\
 O_p\left(n^{-1/2}\right), && \mbox{if $p> a+d/4$}.\\
 \end{array}\right.
\end{align}
\item Severely ill-posed case: choosing
\begin{align*}
J\sim \left(\log n-\frac{4p+d}{2a}\log\log n\right)^{d/a}
\end{align*}
implies
\begin{align}\label{rate:severe}
\widehat f_J-f(h_0)
=O_p\left((\log n)^{-2p/a}\right).
\end{align}
\end{enumerate}
\end{theorem}

Theorem \ref{thm:rate:quad:fctl} presents an upper bound on the convergence rates of $\widehat f_J$ to $f(h_0)$. When the sieve dimension $J$ is chosen optimally, the convergence rate \eqref{rate:mild} coincides with the minimax lower bound in \cite[Theorem C.1]{ChenChristensen2017} for the mildly ill-posed case, while the convergence rate \eqref{rate:severe} coincides with the minimax lower bound in \cite[Theorem C.1]{ChenChristensen2017} for the severely ill-posed case.
Moreover, within the mildly ill-posed case, depending on the smoothness of $h_0$ relatively to the dimension of $X$ and the degree of mildly ill-posedness $a$, either the first or the second variance term in \eqref{rate:extend:cond} dominates, which leads to the so-called elbow phenomenon: the regular case with a parametric rate of $n^{-1/2}$ when $p> a+d/4$; and the irregular case with a nonparametric rate when $p\leq a+d/4$.
In particular, Theorem \ref{thm:rate:quad:fctl} shows that the simple leave-one-out estimator $\widehat f_J$ is minimax rate optimal provided that the sieve dimension $J$ is chosen optimally.

\cite[Theorem C.1]{ChenChristensen2017} actually established lower bound for estimating a quadratic functional of a derivative of $h_0$ in a NPIV model as well. Using Fourier, spline and wavelet bases, we can easily show that our simple leave-one-out, sieve NPIV estimator of the quadratic functional of a derivative of $h_0$ also achieve the lower bound, and hence is minimax rate-optimal. We do not present such a result here since it is a very minor extension of Theorem \ref{thm:rate:quad:fctl}.

\section{Rate Adaptive Estimation}\label{sec:adapt}

The minimax rate of convergence depends on the optimal choice of sieve dimension $J$, which depends on the unknown smoothness $p$ of the true NPIV function $h_0$ and the unknown degree of ill-posedness. In this section we propose a data-driven choice of the sieve dimension $J$ based on a modified Lepski method; see \cite{Lepski90}, \cite{lepski1997} and \cite{lepski1997optimal} for early development of this popular method.

In this section we follow \cite{chen2021} and let $\Psi_J$ be a tensor-product Cohen-Daubechies-Vial (CDV) wavelet (see, e.g., chapter 4.3.5 of \cite{nicklbook}) or dyadic B-spline sieve (see, e.g., Appendix A.1 of \cite{chen2021}) for $\mathcal H_2(p,L)$. Let $\mathcal T$ denote the set of possible sieve dimensions $J$. For example for (order $r$) B-splines, $\mathcal T =\{J=(2^l + r -1)^d:l\in \mathbb{N}\cup\{0\}\}$.
Since $\widehat f_J$ is based on a sieve NPIV estimator, we can simply use a random index set $\widehat{\mathcal I}$ that is proposed in \cite{chen2021} for their sup-norm rate adaptive sieve NPIV estimation of $h_0$:
\begin{align*}
\widehat{\mathcal I}=\{J\in\mathcal T: 0.1(\log \widehat J_{\max})^2\leq J\leq \widehat J_{\max}\},
\end{align*}
where
\begin{align}\label{def:hat_J_max}
\widehat{J}_{\max }=\min \left\{J \in \mathcal T: \widehat{s}_{J}^{-1} J \sqrt{\log J} \leq 10 \sqrt{n} <\widehat{s}_{J^+}^{-1} J^+ \sqrt{\log J^+}\right\},
\end{align}
$\widehat s_J$ is the smallest singular value of $(B'B/n)^{-1/2}(B'\Psi/n)G_\mu^{-1/2}$, and $J^+=\min\{j\in\mathcal T:\,j > J\}$.

We define our data driven choice $\widehat J$ of ``optimal'' sieve dimension for estimating $f(h_0)$ as follows:
\begin{equation}\label{def:J:hat}
\widehat J=\min\Big\{J\in\widehat{\mathcal I}:\, |\widehat f_J-\widehat f_{J'}|\leq c_0 (\widehat V(J)+\widehat V(J'))\text{ for all } J'\in\widehat{\mathcal I} \text{ with }J'> J\Big\}
\end{equation}
for some constant $c_0>0$ and
\begin{align}\label{def:V:hat}
\widehat V(J)&=\frac{\sqrt{J(\log n)}}{n\,\widehat s_J^{2} } \vee \frac{1}{\sqrt n},
\end{align}
where $a\vee b:=\max\{a,b\}$. The random index set $\widehat{\mathcal I}$ is used to compute our data driven choice \eqref{def:J:hat} since the unknown measure of ill-posedness $\tau_J$ is estimated by $\widehat s_J^{-1}$.

We introduce a non-random index set $\mathcal I=\{J\in\mathcal T:\, J\leq \overline J\}$, where $\overline {J}=\sup \left\{J\in\mathcal T: \tau_{J} J \sqrt{(\log J) / n} \leq \bar{c} \right\}$ for some sufficiently large constant $\bar{c}>0$.
Let $\mathcal B=\{h\in L_\mu^2: \|h\|_\infty\leq L\}$ and $\overline p>\underline p\geq 3d/4$. The following assumption strengthens some conditions imposed in the previous section.

\begin{A}\label{A:adapt}
(i) $\sup_{h\in\mathcal H_2(p,L)\cap\mathcal B}\|h-\Pi_J h\|_{\mu} \leq c J^{-p/d}$ for some finite constant $c>0$ for all $p\in[\underline p,\overline p]$, with $\Psi_J$ being CDV wavelet or dyadic B-spline basis;
(ii) $\sup_{w\in\mathcal W}\E[Y^4|W=w]\leq \overline\sigma_Y^4<\infty$;
(iii) Assumptions \ref{A:link}(ii) and \ref{A:per} hold for all $J\in\mathcal I$.
\end{A}

The next result establishes an upper bound for the adaptive estimator $\widehat f_{\widehat J}$.

\begin{theorem}\label{thm:adaptive}
Let Assumptions \ref{A:ident}(i)(iii), \ref{A:link}(i), and \ref{A:adapt} hold.
 Then, we have in the
\begin{enumerate}
\item mildly ill-posed case:
\begin{align}\label{est:adapt:mild}
\sup_{p\in[\underline p,\overline p]}\sup_{h_0\in\mathcal H_2(p,L)\cap\mathcal B}\Pr_{h_0} \left(\big|\widehat{f}_{\widehat J}-f(h_0)\big|>C_1 r_n\right)=o(1)
\end{align}
for some constant $C_1>0$ and where
\begin{align*}
r_n=
\left\{\begin{array}{lcl}
\big(\sqrt{\log n}/n\big)^{4p/(4(p+a)+d)}, && \mbox{if $p\leq a+d/4$},\\
n^{-1/2}, && \mbox{if $p> a+d/4$}.\\
 \end{array}\right.
\end{align*}
\item severely ill-posed case:
\begin{align}\label{est:adapt:severe}
\sup_{p\in[\underline p,\overline p]}\sup_{h_0\in\mathcal H_2(p,L)\cap\mathcal B}\Pr_{h_0} \left(\big|\widehat{f}_{\widehat J}-f(h_0)\big|>C_2 (\log n)^{-2p/a}\right)=o(1)
\end{align}
for some constant $C_2>0$.
\end{enumerate}
\end{theorem}

Theorem \ref{thm:adaptive} shows that our data-driven choice of the key sieve dimension can lead to fully adaptive rate-optimal estimation of $f(h_0)$ for both the severely ill-posed case and the regular mildly ill-posed case, while it has to pay a price of an extra $\sqrt{\log n}$ factor for the irregular mildly ill-posed case (i.e., when $p\leq a+d/4$). We note that when $a =0$ in the mildly ill-posed case, the NPIV model (\ref{npiv:model}) becomes the regression model with $X=W$. Thus our result is in agreement with the theory in \cite{efromovich1996}, which showed that one must pay a factor of $\sqrt{\log n}$ penalty in adaptive estimation of a quadratic functional in a Gaussian white noise model when $p \leq d/4$.

In adaptive estimation of a nonparametric regression function $\E[Y|X=\cdot ]=h(\cdot)$, it is known that Lepski method has the tendency of choosing small sieve dimension, and hence may not perform well in empirical work. We wish to point out that due to the ill-posedness of the NPIV model (\ref{npiv:model}), the optimal sieve dimension for estimating $f(h_0)$ is smaller than the optimal sieve dimension for estimating $f(\E[Y|X=\cdot])$. Therefore, we suspect that our simple adaptive estimator of a quadratic functional of a NPIV function will perform well in finite samples.

\section{Conclusion and Extensions}\label{sec:end}

In this paper we first show that a simple leave-one-out sieve NPIV estimator of the quadratic functional $f(h_0)$ is minimax rate optimal. We then propose an adaptive leave-one-out sieve NPIV estimator of the $f(h_0)$ based on a modified Lepski method to account for the unknown degree of ill-posedness. We show that the adaptive estimator achieves the minimax optimal rate for the severely ill-posed case and for the regular mildly ill-posed case, while a multiplicative $\sqrt{\log n}$ term is the price to pay for the irregular mildly ill-posed NPIV problem.

Like all existing work using Lepski method, implementation of our data-driven choice relies on a calibration constant. To improve finite sample performance over the original Lepski method, \cite{spokoiny2009parameter} suggest a propagation approach, \cite*{CCK} and \cite{spokoiny2019bootstrap} propose bootstrap calibrations in kernel density estimation and in linear regressions with Gaussian errors respectively. \cite{chen2021} propose a bootstrap implementation of a modified Lepski method in their minimax adaptive sup-norm estimation in a NPIV model, and show its good performance in finite samples. Their bootstrap implementation can be easily extended to calibrate the constant in our adaptive estimation of the quadratic functional in a NPIV model. We leave this to future refinement.

Our results can be extended in several directions. First, we can relax the Sobolev ball assumption imposed on $h_0$ in the NPIV model. We can let the NPIV function $h_0$ belong to a bump algebra space. The result by \cite{collier2017} on minimax estimation of a quadratic functional under sparsity constraints can be useful for this extension. Second, we focus on adaptive estimation of a quadratic functional of the NPIV function $h_0$ in this paper. There are works on minimax-rate estimation and adaptive estimation for more general smooth nonlinear functionals of densities and of nonparametric regressions; see, e.g., \cite{birge1995}, \cite*{TT2021} and the references therein. We can combine our approach here with those in the literature for extensions to other smooth nonlinear functionals of the NPIV function $h_0$. Such an extension will allow for adaptive minimax estimation of nonlinear policy functionals in economics and modern causal inference.

\bigskip
\appendix

\section{Proofs of Results in Section \ref{sec:minimax}}\label{appendix:minimax}
Recall the 2SLS projection of $h$ onto $\Psi_J$ is given by:
\[
Q_J h(x)=\psi^J(x)'[S'G_b^{-1}S]^{-1}S'G_b^{-1}\E[b^{K(J)}(W)h(X)]
=\psi^J(x)'A\E[b^{K(J)}(W)h(X)].
\]
For a $r\times c$ matrix $M$ with $r \leq  c$ and full row rank $r$ we let $M_l^-$ denote its left pseudoinverse, namely $(M'M)^-M'$.
Let $\widetilde \psi^J=G_\mu^{-1/2}\psi^J$ and $\widetilde b^K=G_b^{-1/2}b^K$.
Thus, we have $AG_{b}^{1/2}=(G_b^{-1/2}S)_l^-$ and
\[
G_{\mu}^{1/2}AG_{b}^{1/2}=(G_b^{-1/2}SG_\mu^{-1/2})_l^-.
\]
In particular, we can write
\begin{align*}
Q_J h(x)&=\psi^J(x)'(G_b^{-1/2}S)_l^-\E[\widetilde b^{K(J)}(W)h(X)]\\
&=\widetilde \psi^J(x)'(G_b^{-1/2}SG_\mu^{-1/2})_l^-\E[\widetilde b^{K(J)}(W)h(X)].
\end{align*}

The minimal or maximal eigenvalue of a quadratic matrix $M$ is denoted by $\lambda_{\min}(M)$ or $\lambda_{\max}(M)$. Recall that
\[
\widehat{f_J}
=\frac{1}{n(n-1)}\sum_{i\neq i'}Y_i Y_{i'}b^K(W_i)'\widehat A'G_\mu\widehat A\,b^K(W_{i'}).
\]

\begin{proof}[\textsc{Proof of Theorem \ref{thm:rate:quad:fctl}.}]

Proof of Result \eqref{rate:extend:cond}. Note that
\begin{align*}
f(Q_Jh_0)&=\int \Big(\psi^J(x)'(G_b^{-1/2} S)_l^-\E[\widetilde b^K(W)h_0(X)]\Big)^2\mu(x)dx\\
&=\big\|G_{\mu}^{1/2}(G_b^{-1/2} S)_l^-\E[\widetilde b^K(W)h_0(X)]\big\|^2= \|\E[V^J]\|^2
\end{align*}
using the notation $V_{i}^J=Y_iG_\mu^{1/2} A b^K(W_i)$.
Thus, the definition of the estimator $\widehat f_J$ implies
\begin{align}
\widehat f_J-f(Q_Jh_0)&= \frac{1}{n(n-1)}\sum_{j=1}^J\sum_{i\neq i'}\big( V_{ij}V_{i'j}- \E[V_{1j}]^2\big)\label{eq1:proof:upper}\\
&\quad +\frac{1}{n(n-1)}\sum_{i\neq i'} Y_iY_{i'}b^K(W_i)'\Big(A' G_\mu A-\widehat A'G_\mu  \widehat A\Big)b^K(W_{i'}),\label{eq2:proof:upper}
\end{align}
where we bound both summands on the right hand side separately in the following.
Consider the  summand in \eqref{eq1:proof:upper}, we observe
\begin{align*}
&\E\Big|\sum_{j=1}^J\sum_{i\neq i'}\big( V_{ij}V_{i'j}- \E[V_{1j}]^2\big)\Big|^2\\
&=2n(n-1)(n-2)\underbrace{\sum_{j,j'=1}^J\E\Big[\big( V_{1j}V_{2j}- \E[V_{1j}]^2\big)\big( V_{3j'}V_{2j'}- \E[V_{1j'}]^2\big)\Big]}_{I}\\
&+n(n-1)\underbrace{\sum_{j,j'=1}^J\E\Big[\big( V_{1j}V_{2j}- \E[V_{1j}]^2\big)\big( V_{1j'}V_{2j'}- \E[V_{1j'}]^2\big)\Big]}_{II}.
\end{align*}
By Assumption \ref{A:ident}(ii) it holds $\sup_{w\in\mathcal W}\E[Y^2|W=w]\leq \overline\sigma_Y^2$, which together with  \cite[Lemma E.7]{BC2020} implies $\lambda_{\max}\big(\Var(Y \widetilde b^K(W)\big)\leq \overline\sigma_Y^2$.
To bound the summand $I$ we observe that
\begin{align*}
I
&=\sum_{j,j'=1}^J\E[V_{1j}]\E[V_{1j'}]\Cov(V_{1j},V_{1j'})=\E[V_{1}^J]'\Cov(V_{1}^J,V_{1}^J)\E[V_{1}^J]\\
&\leq \lambda_{\max}\big(\Var(Y  \widetilde b^K(W))\big)\big\|(G_b^{-1/2} SG_\mu^{-1/2})_l^-\E[V_{1}^J]\big\|^2\\
&=\overline\sigma_Y^2\Big\|\langle  Q_J h_0,\psi^J\rangle_\mu' (G_b^{-1/2}S)_l^-\Big\|^2
\end{align*}
by using the notation $V_{i}^J=Y_i(G_b^{-1/2} SG_\mu^{-1/2})_l^- \widetilde b^K(W_i)$.
Consider $II$. We observe
\begin{align*}
II&= n(n-1)\sum_{j,j'=1}^J\E[V_{1j}V_{1j'}]^2- n(n-1)\Big(\sum_{j=1}^J\E[V_{1j}]^2\Big)^2\\
&\leq  n(n-1)\sum_{j,j'=1}^J\E[V_{1j}V_{1j'}]^2\leq 2\overline\sigma_Y^2 n(n-1)s_J^{-4}J
\end{align*}
where the last inequality stems from \cite[Lemma E.1]{BC2020} together with $\sup_{w\in\mathcal W}\E[Y^2|W=w]\leq \overline\sigma_Y^2$.
Consequently, we obtain
\begin{equation}\label{upper:bound:mse}
\E\Big|\frac{1}{n(n-1)}\sum_{j=1}^J\sum_{i\neq i'}\big( V_{ij}V_{i'j}- \E[V_{1j}]^2\big)\Big|^2
\leq 4\overline\sigma_Y^4\left(\frac{1}{n}\big\|\langle Q_J h_0,\psi^J\rangle_\mu'(G_b^{-1/2}S)_l^-\big\|^2+\frac{ J}{n^2s_J^4}\right).
\end{equation}
The second summand in \eqref{eq2:proof:upper} can be bounded following the same proof as that of \cite[Lemma E.4]{BC2020} (replacing their $(Y_i-h_0(X_i))$ with our $Y_i$ and our Assumption \ref{A:ident}(ii)), which yields
\begin{align*}
\widehat f_J-f(Q_Jh_0)=O_p\left(\frac{1}{\sqrt{n}}\big\|\langle Q_J h_0,\psi^J\rangle_\mu'(G_b^{-1/2}S)_l^-\big\|+\frac{\sqrt J}{ns_J^2}\right).
\end{align*}
Next, by the definition of $Q_J$ we have:
$\langle Q_Jh_0,\widetilde \psi^J\rangle_\mu=(G_b^{-1/2}SG_\mu^{-1/2})^-_l\E[\widetilde b^{K(J)}(W)h_0(X)]$. Thus, we have
\begin{align*}
\big\|\langle Q_Jh_0,\psi^J\rangle_\mu'(G_b^{-1/2} S)_l^-\big\|
&\leq \big\|\langle h_0, \psi^J\rangle_\mu'(G_b^{-1/2}S)_l^-\big\|\\
&\quad +s_J^{-2}\big\|\E[\widetilde b^{K(J)}(W) (h_0(X)-\Pi_Jh_0(X))]\big\|,
\end{align*}
By inequality \eqref{s_J:upper1} and Assumption \ref{A:link}(ii), we have
\begin{align*}
s_J^{-2}\big\|\E[\widetilde b^{K(J)}(W) (h_0(X)-\Pi_Jh_0(X))]\big\|
&= O\Big(\tau_J^2\|\Pi_K T(h_0-\Pi_J h_0 )\|_{L^2(W)}\Big)\\
&= O\big(\tau_J\|h_0-\Pi_J h_0 \|_\mu\big).
\end{align*}
It remains to evaluate
\begin{align*}
f(Q_Jh_0)-f(h_0)= \|Q_J h_0\|_\mu^2-\Big[\|\Pi_J h_0\|_\mu^2 +
\|h_0-\Pi_J h_0\|_\mu^2\Big].
\end{align*}
Consider the first summand on the right hand side.
There exist unitary matrices $M_1$, $M_2$ with $\dot  b^K:=M_1 \widetilde b^K$ and $\dot  \psi^J:=M_2 \widetilde \psi^J$ such that
$\E[\dot  b^{K(J)}(W)\dot  \psi^J(X)']$ has an upper $J\times J$
matrix $\textsl{diag}(s_1,\dots,s_J)$ and is zero otherwise.
We thus derive
\begin{align*}
\|Q_J h_0\|_\mu^2 &=\Big\|(G_b^{-1/2}SG_\mu^{-1/2})^-_l\E[\widetilde b^{K(J)}(W)h_0(X)]\Big\|^2\\
&=\sum_{j=1}^Js_j^{-2} \E[\dot b_j(W) h_0(X)]^2=\sum_{j=1}^J \langle h_0,\dot\psi_j\rangle_\mu^2=\|\Pi_J h_0\|_\mu^2,
\end{align*}
and hence $f(Q_Jh_0)-f(h_0)= -\|h_0-\Pi_J h_0\|_\mu^2$.
This completes the proof of Result \eqref{rate:extend:cond}.

For the proofs of Results \eqref{rate:mild} and \eqref{rate:severe}, we note that 
$h_0 \in \mathcal H_2(p, L)$ implies
\begin{align*}
\|h_0-\Pi_J h_0\|_\mu \leq LJ^{-p/d}.
\end{align*}
Moreover, by inequality \eqref{s_J:upper1} and Assumption \ref{A:per} we have:
\begin{align*}
\big\|\langle h_0,\psi^J\rangle_\mu'(G_b^{-1/2} S)_l^-\big\|&=\big\|\langle h_0,\widetilde \psi^J\rangle_\mu'(G_b^{-1/2}S'G_\mu^{-1/2})_l^-\big\|
\leq Dc_\tau^{-1}\sqrt{\sum_{j=1}^J\tau_j^{2}\langle h_0,\widetilde \psi_j\rangle_\mu^2}.
\end{align*}
These bounds are used below to derive the concrete rates of convergence in the mildly and severely ill-posed regimes.

Proof of Result \eqref{rate:mild} for the mildly ill-posed case.
The choice of $J\sim n^{2d/(4(p+a)+d)}$ implies
\begin{align*}
n^{-2}\tau_J^4 J\sim n^{-2}J^{1+4a/d}\sim n^{-8p/(4(p+a)+d)}
\end{align*}
and for the bias term $J^{-4p/d}\sim n^{-8p/(4(p+a)+d)}$.
We now distinguish between the two regularity cases of the result.
First, consider the case $p\leq a+d/4$, where the mapping $j\mapsto j^{2(a-p)/d+1/2}$ is increasing  and consequently, we observe
\begin{align*}
n^{-1}\sum_{j=1}^J\langle h_0,\widetilde \psi_j\rangle_\mu^2\,\tau_j^2
&\sim n^{-1}\sum_{j=1}^J\langle h_0,\widetilde\psi_j\rangle_\mu^2\, j^{2p/d-1/2} j^{2(a-p)/d+1/2}\\
&\lesssim n^{-1} J^{2(a-p)/d+1/2}  \sim n^{-8p/(4(p+a)+d)}.
\end{align*}
Moreover, we obtain
\begin{align*}
n^{-1}\tau_J^2 J^{-2p/d}
&\sim n^{-1}J^{2(a-p)/d}\lesssim n^{-8p/(4(p+a)+d)}.
\end{align*}
Finally, it remains to consider the case $p> a+d/4$.
In this case, we have that
\begin{align*}
\sum_{j=1}^J\langle h_0,\widetilde\psi_j\rangle_\mu^2\,\tau_J^2&\lesssim\sum_{j=1}^J\langle h_0,\widetilde\psi_j\rangle_\mu^2 j^{2p/d}=O(1)
\end{align*}
and consequently, the second variance term satisfies $n^{-1}\big\|\langle Q_Jh_0,\psi^J\rangle_\mu(G_b^{-1/2}S)_l^-\big\|^2=O(n^{-1})$ which is the dominating rate and thus, completes the proof of the result.

Proof of Result \eqref{rate:severe} for the severely ill-posed case.  The choice of
\begin{align*}
J\sim \left(\log n-\frac{4p+d}{2a}\log\log n\right)^{d/a}
\end{align*}
 implies
\begin{equation*}
n^{-2}\tau_J^4 J\sim n^{-2}J\exp(2J^{a/d})
\sim \left(\log n-\frac{4p+d}{2a}\log\log n\right)^{d/a} (\log n)^{-(4p+d)/a}
\sim(\log n)^{-4p/a}.
\end{equation*}
We further analyze for the bias part
\begin{align*}
J^{-4p/d}&\sim \left(\log n-\frac{4p+d}{2a}\log\log n\right)^{-4p/a}
\sim (\log n)^{-4p/a}.
\end{align*}
Moreover, since the mapping $j\mapsto j^{-2p/d}\exp(j^{a/d})$ is increasing  we obtain
\begin{align*}
n^{-1}\sum_{j=1}^J\langle h_0,\widetilde\psi_j\rangle_\mu^2\,\tau_j^2
&\sim n^{-1}\sum_{j=1}^J\langle h_0,\widetilde\psi_j\rangle_\mu^2j^{2p/d}j^{-2p/d} \exp(j^{a/d})\\
&\lesssim n^{-1} \exp(J^{a/d})  J^{-2p/d}
\sim \left(\log n\right)^{-2p/a} (\log n)^{-(2p+d)/a} \lesssim (\log n)^{-4p/a}
\end{align*}
and finally
\begin{align*}
n^{-1}\tau_J^2 J^{-2p/d}&\sim n^{-1} \exp(J^{a/d}) J^{-2p/d}\lesssim (\log n)^{-4p/a},
\end{align*}
which shows the result.
\end{proof}

\section{Proofs of Results in Section \ref{sec:adapt}}\label{appendix:adapt}
We denote $\mathcal H=\bigcup_{p\in[\underline p,\overline p]}\mathcal H_2(p,L)\cap \mathcal B$ and recall that $\mathcal B=\mathcal B(L)=\{h:\|h\|_\infty<L\}$.
Below, we make use of the notation
\begin{align*}
\widehat{\mathcal J}=\Big\{J\in\widehat{\mathcal I}:\, |\widehat f_J-\widehat f_{J'}|\leq c_0 (\widehat V(J)+\widehat V(J'))\text{ for all } J'\in\widehat{\mathcal I} \text{ with }J'> J\Big\}
\end{align*}
and recall the definition $\widehat{\mathcal I}=\{J\in\mathcal T: 0.1(\log \widehat J_{\max})^2\leq J\leq \widehat J_{\max}\}$.
We denote
\begin{align*}
\overline J(c)=\sup\{J\in\mathcal T:\tau_JJ \sqrt{(\log J)/n}\leq c\}
\end{align*}
for some constant $c>0$.
The oracle choice of the dimension parameter is given by
\begin{equation}\label{def:oracle}
J_0=J_0(p,c_0)=\sup \Big\{J\in\mathcal T:\, V(J)\leq c_0 J^{-2p/d}\Big\}, V(J)=\tau_J^2 \frac{\sqrt{J(\log n)}}{n}\vee \frac{1}{\sqrt n}
\end{equation}
for some constant $c_0>0$.
We introduce the set $$\mathcal E_n^*=\{J_0\in\widehat{\mathcal J}\}\cap\{|\widehat s_J-s_J|\leq \eta s_J\text{ for all }J\in\mathcal I\}$$ for some $\eta\in(0,1)$.

\begin{proof}[\textsc{Proof of Theorem \ref{thm:adaptive}.}]
Proof of Result \eqref{est:adapt:mild} for the mildly ill-posed case.
Due to \cite[Lemma B.5]{chen2021}
we have $\overline J(c_1)\leq \widehat J_{\max}\leq \overline J(c_2)$ for some constants $c_1,c_2>0$ on $\mathcal E_n^*$.
The definition $\widehat J=\min_{J\in\widehat{\mathcal J}}J$ implies $\widehat J\leq J_0$ on the set $\mathcal E_n^*$ and hence, we obtain
\begin{align*}
\big| \widehat f_{\widehat J}-f(h_0)\big|\1_{\mathcal E_n^*}
&\leq \big| \widehat f_{\widehat J}-\widehat f_{J_0}\big|\1_{\mathcal E_n^*}+| \widehat f_{J_0}-f(h_0)|\1_{\mathcal E_n^*}\\
&\leq c_0\Big(\widehat V(\widehat J)+\widehat V(J_0)\Big)\1_{\mathcal E_n^*} +| \widehat f_{J_0}-f(h_0)|.
\end{align*}
On the set $\mathcal E_n^*$, we have $|\widehat s_J-s_J|\leq \eta s_J$, for some $\eta\in(0,1)$, which implies $\widehat s_J^{-2}\leq s_J^{-2}(1-\eta)^{-2}$ and thus, by the definition of $\widehat V(\cdot)$ in \eqref{def:V:hat} we have
\begin{equation*}
\big| \widehat f_{\widehat J}-f(h_0)\big|\1_{\mathcal E_n^*}\leq c_0 (1-\eta)^{-2} \Bigg(\Big(s_{\widehat J}^{-2}\sqrt{\widehat J}+s_{J_0}^{-2}\sqrt{J_0}\Big)\1_{\mathcal E_n^*}\frac{\sqrt{\log n}}{n}\Bigg)\vee\frac{1}{\sqrt n}+| \widehat f_{J_0}-f(h_0)|.
\end{equation*}
Using inequality \eqref{s_J:upper1} together with Assumption \ref{A:link}(i) yields $s_J^{-2}\leq c_\tau\tau_J^2$ for all $J$, see inequality \eqref{s_J:upper1}. Consequently, from the definition of $V(\cdot)$ in \eqref{def:oracle} we infer:
\begin{align*}
\big| \widehat f_{\widehat J}-f(h_0)\big|\1_{\mathcal E_n^*}&\leq \frac{c_0\,c_\tau}{ (1-\eta)^{2}}\Bigg(\Big(\tau_{\widehat J}^{2}\sqrt{\widehat J}+\tau_{J_0}^2\sqrt{J_0}\Big)\1_{\mathcal E_n^*}\frac{\sqrt{\log n}}{n}\Bigg)\vee\frac{1}{\sqrt n}+| \widehat f_{J_0}-f(h_0)|\\
&\leq \frac{c_0\,c_\tau}{ (1-\eta)^{2}}\Big(V(\widehat J)+V(J_0)\Big)\1_{\mathcal E_n^*}+| \widehat f_{J_0}-f(h_0)|\\
&\leq \frac{2c_0\,c_\tau}{ (1-\eta)^{2}}V(J_0)+| \widehat f_{J_0}-f(h_0)|
\end{align*}
for $n$ sufficiently large,
where the last inequality is due to $V(\widehat J)\1_{\mathcal E_n^*}\leq V(J_0)$ since $\widehat J\leq J_0$ on $\mathcal E_n^*$.
By Lemmas \ref{lemma:bounds:tau} and \ref{lemma:adapt} it holds $\Pr(\mathcal E_n^*)=1+o(1)$.

The definition of the oracle choice in \eqref{def:oracle} implies $J_0\sim (n/\sqrt{\log n})^{2d/(4(p+a)+d)}$ in the mildly ill-posed case. Thus, we obtain
\begin{align*}
n^{-2}(\log n)\tau_{J_0}^4 J_0\sim n^{-2}(\log n)J_0^{1+4a/d}\sim (\sqrt{\log n}/n)^{8p/(4(p+a)+d)}
\end{align*}
which coincides with the rate for the bias term.
We now distinguish between the two cases in the mildly ill-posed case.
First, consider the case $p\leq a+d/4$.
In this case, the mapping $j\mapsto j^{2(a-p)/d+1/2}$ is increasing in $j$ and consequently, we observe
\begin{align*}
n^{-1}\sum_{j=1}^{J_0}\tau_j^2 \langle h_0,\widetilde \psi_j\rangle_\mu^2
&\lesssim J_0^{2(a-p)/d+1/2} n^{-1}\lesssim (\sqrt{\log n}/n)^{8p/(4(p+a)+d)}.
\end{align*}
Moreover, using $h_0 \in \mathcal H_2(p, L)$, i.e., $\sum_{j\geq 1}\langle h_0,\widetilde\psi_j\rangle_\mu^2 j^{2p/d}\leq L$, we obtain
\begin{align*}
n^{-1}\tau_{J_0}^2 \sum_{j>J_0}\langle h_0,\widetilde\psi_j\rangle_\mu^2
\lesssim  (\sqrt{\log n}/n)^{8p/(4(p+a)+d)}.
\end{align*}
Finally, it remains to consider the case $p> a+d/4$, where as in the proof of Theorem \ref{thm:rate:quad:fctl} we have $\sum_{j=1}^{J_0}\tau_j^2\langle h_0,\widetilde\psi_j\rangle_\mu^2=O(1)$,
implying $n^{-1}\big\|\langle Q_{J_0}h_0,\psi^{J_0}\rangle_\mu(G_b^{-1/2}S)_l^-\big\|^2=O(n^{-1})$ which is the dominating rate and thus, completes the proof for the mildly ill-posed case.

Proof of Result \eqref{est:adapt:severe} for the severely ill-posed case. We have
\begin{align*}
\big|& \widehat f_{\widehat J}-f(h_0)\big|\1_{\mathcal E_n^*}
\leq \big| \widehat f_{\widehat J}-f(Q_{\widehat J}h_0)\big|\1_{\mathcal E_n^*}+\max_{\overline J(c_1)\leq J\leq \overline J(c_2)}| f(Q_{J}h_0 - h_0)|\1_{\mathcal E_n^*}\\
&\leq  2\overline\sigma_Y^2\Bigg(s_{\overline J(c_2)}^{-2}
\frac{\sqrt{\overline J(c_2)\log\overline J(c_2)}}{n-1}
+\frac{\|\langle  Q_{\overline J(c_2)} h_0,\psi^{\overline J(c_2)} \rangle_\mu' (G_b^{-1/2}S)_l^-\|}{\sqrt n}\Bigg)
+\big(\overline J(c_1)\big)^{-2p/d}
\end{align*}
with probability approaching one by Lemma \ref{lemma:uniform:bound}.
From \cite[Lemma B.2]{chen2021} it holds, in the severely ill-posed case,  $J_0^+=\inf\{J\in\mathcal T: J>J_0\}\geq \overline J(c_1)$ for all $n$ sufficiently large and thus,
by the definition of $\overline J(\cdot)$ we have
\begin{align*}
\big| \widehat f_{\widehat J}-f(h_0)\big|\1_{\mathcal E_n^*}\leq  (2\overline\sigma_Y^2+1)\big(C\overline J(c_2)\big)^{-2p/d}
\end{align*}
with probability approaching one,  using that $\overline J(c_1)\geq C \overline J(c_2)$ for some constant $C>0$. From the definition of $\overline J(\cdot)$ we have $(c\log n)^{d/a}\leq \overline J(c_2)$ for any $c\in(0,1)$ and $n$ sufficiently large. This implies
\begin{align*}
\big| \widehat f_{\widehat J}-f(h_0)\big|\1_{\mathcal E_n^*}=O_p\big((\log n)^{-2p/a}\big),
\end{align*}
which completes the proof.
\end{proof}

\section{Supplementary Lemmas}\label{appendix:lemmas}

We first introduce additional notation. First we consider a U-statistic
\begin{align*}
\mathcal U_{n,1}=\frac{2}{n(n-1)}\sum_{i< i'}R_1(Z_i, Z_{i'})
\end{align*}
where $Z_i=(Y_i,W_i)$  and the kernel $R_1$ is given  by
\begin{align}
R_1(Z_i, Z_{i'})&=Y_i\1_{M_i}b^{K(J)}(W_i)'A' G_\mu Ab^{K(J)}(W_{i'})Y_{i'}\1_{M_{i'}}\nonumber\\
&\qquad - \E[Y_i\1_{M_i}b^{K(J)}(W_i)]'A' G_\mu A \E[b^{K(J)}(W_{i'})Y_{i'}\1_{M_{i'}}]\label{kernel:R1}
\end{align}
where $M_i= \{|Y_i|\leq M_n \}$ with $M_n = J^{-1/4} \sqrt{n}/(\log \overline J)$.
Note that the kernel $R_1$ is a symmetric function such that $\E[R_1(Z_i,z)]=0$ for all $z$.
We also introduce the U-statistic
\begin{align*}
\mathcal U_{n,2}=\frac{2}{n(n-1)}\sum_{i< i'}R_2(Z_i, Z_{i'})
\end{align*}
where the kernel $R_2$ is given by
\begin{align*}
R_2(Z_i, Z_{i'})&=Y_i b^{K(J)}(W_i)'A' G_\mu Ab^{K(J)}(W_{i'})Y_{i'}\1_{M_i^c\cup M_{i'}^c}\\
 &\qquad- \E\big[Y_i b^{K(J)}(W_i)'A' G_\mu A b^{K(J)}(W_{i'})Y_{i'}\1_{M_i^c\cup M_{i'}^c}\big].
\end{align*}

We make use of the following exponential inequality established by \cite{houdre2003}.
\begin{lemma}[\cite{houdre2003}]\label{Lemma:houdre}
Let $U_n$ be a degenerate U-statistic of order 2 with kernel $R$ based on a simple random sample $Z_1, \dots ,Z_n$. Then there exists a generic constant $C>0$, such that
\begin{align*}
\Pr\left(\Big|\sum_{1\leq i< i'\leq n} R(Z_i, Z_{i'})\Big|\geq C\Big(\Lambda_1\sqrt{u} +\Lambda_2 u+\Lambda_3 u^{3/2} +\Lambda_4 u^2\Big)\right)\leq 6\exp(-u)
\end{align*}
where
\begin{align*}
\Lambda_1^2&=\frac{n(n-1)}{2}\E[R^2(Z_1,Z_2)],\\
\Lambda_2&=n\sup_{\|\nu\|_{L^2(Z)}\leq 1, \|\kappa\|_{L^2(Z)}\leq 1} \E[R(Z_1,Z_2)\nu(Z_1)\kappa(Z_2)],\\
\Lambda_3&=\sqrt{n\sup_{z} |\E[R^2(Z_1,z)]|},\\
\Lambda_4&=\sup_{z_1,z_2} |R(z_1,z_2)|.
\end{align*}
\end{lemma}

The next result provides upper bounds for the estimates $\Lambda_1,\Lambda_2,\Lambda_3,\Lambda_4$ when the kernel $R$ coincides with $R_1$ given in \eqref{kernel:R1}.

\begin{lemma}\label{lemma:bounds:lambda}
Let Assumption \ref{A:ident}(ii) be satisfied. Given kernel $R=R_1$, it holds:
\begin{align}
\Lambda_1^2&\leq \frac{\overline\sigma_Y^4  n(n-1)}{2}J s_J^{-4}\label{bound:L_1}\\
\Lambda_2&\leq 2\overline\sigma_Y^2\,n\,s_J^{-2}\label{bound:L_2}\\
\Lambda_3&\leq \overline\sigma_Y^2M_n\sqrt{n\,c_K J}  s_J^{-2}\label{bound:L_3}\\
\Lambda_4&\leq c_K M_n^2 J s_J^{-2}\label{bound:L_4}.
\end{align}
\end{lemma}
\begin{proof}
The result follows from the proofs of Lemma E.1 and Lemma F.2 of \cite{BC2020} using Assumption \ref{A:ident}(ii).
\end{proof}

\begin{lemma}\label{lemma:bounds:tau}
Let Assumption \ref{A:ident}(iii)
hold and $\Psi_J$ be CDV wavelet or dyadic B-spline basis. Then, for any constant $\eta\in(0,1)$ we have
\begin{align*}
\Pr\Big(\sup_{J\in\mathcal I}s_J^{-1}|\widehat s_J-s_J|\leq \eta\Big)\to 1.
\end{align*}
\end{lemma}
\begin{proof}
The result is due to the proof of Lemma C.7 of \cite{chen2021}.
\end{proof}

\begin{lemma}\label{lemma:bounds:hat_J}
Let Assumptions \ref{A:ident}(iii) and \ref{A:link}(i) hold and $\Psi_J$ be CDV wavelet or dyadic B-spline basis. Then, for any constants $c_1,c_2>0$ we have
\begin{align*}
\Pr\Big(\overline J(c_1)\leq \widehat J_{\max}\leq \overline J(c_2)\Big)\to 1.
\end{align*}
\end{lemma}
\begin{proof}
The result is due to \cite[Lemma B.5]{chen2021}.
\end{proof}

Below, for simplicity of notation, we denote $\overline J:=\overline J(c)$ for some constant $c>0$ and we make use of the notation
\begin{align*}
c_n(J):=  2\overline\sigma_Y^2s_J^{-2} \, \frac{\sqrt{J\log \overline J}}{n-1}+\|\langle  Q_J h_0,\psi^J\rangle_\mu' (G_b^{-1/2}S)_l^-\|/\sqrt n
\end{align*}
\begin{lemma}\label{lemma:uniform:bound}
Let Assumptions \ref{A:ident}(i)(iii), \ref{A:link}(i), and \ref{A:adapt} be satisfied.
Then, we have
\begin{align}\label{lemma:adapt:ineq}
\inf_{h_0\in\mathcal H}\Pr_{h_0}\Big(\big| \widehat f_J-f(Q_Jh_0)\big|\leq c_n(J) \quad\forall J\in\mathcal I\Big)\to 1.
\end{align}
\end{lemma}
%
\begin{proof}
 First, observe that by making use of Assumption \ref{A:adapt} it holds $c_n(J)=o(1)$ uniformly in $J\in\mathcal I$.
We make use of the decomposition
\begin{align*}
&\widehat f_J-f(Q_Jh_0)\\
&=\frac{2}{n(n-1)}\sum_{i< i'} \big(Y_ib^{K(J)}(W_i) - \E[Yb^{K(J)}(W)]\big)'A' G_\mu A\big(Y_{i'}b^{K(J)}(W_{i'}) - \E[Yb^{K(J)}(W)]\big)\\
&\quad +\frac{4}{n}\sum_{i} \E[Yb^{K(J)}(W)]'A' G_\mu A\big(Y_ib^{K(J)}(W_i) - \E[Yb^{K(J)}(W)]\big)\\
&\quad +\frac{2}{n(n-1)}\sum_{i< i'} Y_iY_{i'}b^{K(J)}(W_i)'\big(A' G_\mu A-\widehat A' G_\mu \widehat A\big)b^{K(J)}(W_{i'}).
\end{align*}
Using the U-statistic notation introduced at the beginning of this section we obtain
\begin{flalign*}
\Pr_{h_0}&\left( \max_{J\in\mathcal I}\Big\{c_n(J)^{-1}\big| \widehat f_J-f(Q_Jh_0)\big|\Big\}>1\right)\leq \Pr\left( \max_{J\in\mathcal I}\Big\{c_n(J)^{-1}\big|\mathcal U_{n,1}(J)\big|\Big\}>\frac{1}{4}\right)&\\
&+\Pr_{h_0}\left( \max_{J\in\mathcal I}\Big\{c_n(J)^{-1}\big|\mathcal U_{n,2}(J)\big|\Big\}>\frac{1}{4}\right)\\
&+\Pr_{h_0}\left( \max_{J\in\mathcal I}\Big\{c_n(J)^{-1}\big|\frac{4}{n}\sum_{i} \E[Yb^{K(J)}(W)]'A' G_\mu A\big(Y_ib^{K(J)}(W_i) - \E[Yb^{K(J)}(W)]\big)\big|\Big\}>\frac{1}{4}\right)\\
&+\Pr_{h_0}\left( \max_{J\in\mathcal I}\Big\{c_n(J)^{-1}\big|\frac{2}{n(n-1)}\sum_{i< i'} Y_iY_{i'}b^{K(J)}(W_i)'\big(A' G_\mu A-\widehat A' G_\mu \widehat A\big)b^{K(J)}(W_{i'})\big|\Big\}>\frac{1}{4}\right)\\
&=:T_1+T_2+T_3+T_4.
\end{flalign*}
Consider $T_1$.
We make use of U-statistics deviation results. To do so, consider $\Lambda_1,\ldots,\Lambda_4$ as given in  Lemma \ref{Lemma:houdre}. From  Lemma \ref{lemma:bounds:lambda} we infer with  $u=2\log \overline J$ and $M_n=J^{-1/4} \sqrt{n}/(\log \overline J)$ that for all $J\leq \overline J$ we  have
\begin{align*}
\Lambda_1\sqrt{u} +\Lambda_2 u+&\Lambda_3 u^{3/2} +\Lambda_4 u^2\\
&\leq
\Lambda_1\sqrt{2\log \overline J} +2\Lambda_2 \log \overline J+\Lambda_3 (2\log \overline J)^{3/2} +4\Lambda_4 (\log \overline J)^2\\
&\leq \overline\sigma_Y^2 n s_J^{-2}\sqrt{J\log \overline J} +4\overline\sigma_Y^2n s_J^{-2} \log \overline J+\overline\sigma_Y^2n  s_J^{-2} \,J^{1/4}\sqrt{2\log \overline J}+4 n s_J^{-2}\sqrt J
\end{align*}
for $n$ sufficiently large.
Hence, we obtain for $n$ sufficiently large
\begin{align*}
\Lambda_1\sqrt{u} +\Lambda_2 u+\Lambda_3 u^{3/2} +\Lambda_4 u^2
\leq 2\overline\sigma_Y^2 n s_J^{-2}\sqrt{J\log \overline J}\leq \frac{n(n-1)}{2} c_n(J)
\end{align*}
by the definition of $c_n(J)$ and Lemma \ref{Lemma:houdre} with $u=2\log \overline J$  yields
\begin{equation*}
T_1\leq \sum_{J\in\mathcal I}\Pr_{h_0}\left(\Big|\sum_{ i< i'} R_1(Z_i, Z_{i'})\Big|\geq \frac{n(n-1)}{2}c_n(J)\right)\leq 6\exp\left((\log \overline J)-2(\log \overline J)\right)\leq \frac{6}{\overline J}
\end{equation*}
and thus, $T_1=o(1)$ since $\overline J$ diverges.
Consider $T_2$.   Markov's inequality together with $\#\mathcal I\leq \log_2(\overline J)\leq 2\log(\overline J)$ yield by following the derivation  of the upper bound \eqref{upper:bound:mse}:
\begin{align*}
T_2&\leq \sum_{J\in\mathcal I}c_n(J)^{-1}\sqrt{\E_{h_0}|\mathcal U_{n,2}(J)|^2}\leq 2\log(\overline J)\max_{J\in\mathcal I}c_n(J)^{-1}\sqrt{\E_{h_0}|\mathcal U_{n,2}(J)|^2}\\
&=O\Big(n^{-1/2}\log(\overline J) \max_{J\in\mathcal I} \frac{\sqrt J}{M_n^{2}c_n(J)s_J^{2}}\Big)
=O\Big(n^{-1/2}\log(\overline J) ^{3/2}\sqrt{\overline J}\Big)=o(1),
\end{align*}
using that $M_n^{-2}=\sqrt J (\log \overline J)^2/n$ and $(c_n(J)s_J^{2})^{-1}\leq n/\sqrt{J\log\overline J}$.
Lemma \ref{lemma:linear:bound} implies $T_3=o(1)$.
Consider $T_4$. We have that
\begin{equation*}
\max_{J\in\mathcal I}\Big\{n^{-1}\Big(\log(J)\sum_{j=1}^Js_j^{-4}\Big)^{-1/2}\sum_{i< i'} Y_iY_{i'}b^{K(J)}(W_i)'\big(A' G_\mu A-\widehat A' G_\mu \widehat A\big)b^{K(J)}(W_{i'})\big|\Big\}=o_p(1)
\end{equation*}
by following Lemma E.5 of \cite{BC2020} (with their $Y_i-h_0(X_i)$ and $\textsl v_J$ replaced by our $Y_i$ and $(\sum_{j=1}^Js_j^{-4})^{1/2}$ respectively, and, in particular, we do not need to impose a lower bound on $\E[Y^2|W]$ since our estimator is un-standardized.) The previous equation implies $T_4=o(1)$.
\end{proof}

\begin{lemma}\label{lemma:linear:bound}
 Let Assumptions \ref{A:ident}(i)(iii), \ref{A:link}(i), and \ref{A:adapt} be satisfied. Then, there exists a constant $C>0$ such that
$$
\inf _{h_0 \in \mathcal{H}} \mathbb{P}_{h_{0}}\left(\sup _{J \in\mathcal I}\Big(c_n^{-1}(J)\Big|\frac{1}{n}\sum_{i} Y_ia_J(W_i) - \E[Ya_J(W)]\Big| \Big)\leq C\right) \rightarrow 1,
$$
where $a_J(w)=\langle  Q_J h_0,\psi^J\rangle_\mu' (G_b^{-1/2}S)_l^-\widetilde b^{K(J)}(w)$.
\end{lemma}
\begin{proof}
The result follows  by an application of Talagrand's inequality analogously to the proof of \cite[Lemma C.2]{chen2021}, based on the following upper bounds:
\begin{align*}
\E\left|Y_ia_J(W_i) - \E[Ya_J(W)]\right|^2&\leq \E\left|Ya_J(W)\right|^2\leq \sigma_Y^2 \|\langle  Q_J h_0,\psi^J\rangle_\mu' (G_b^{-1/2}S)_l^-\|^2
\end{align*}
by using Assumption \ref{A:adapt}(ii) and
\begin{align*}
\left|Y_ia_J(W_i) - \E[Ya_J(W)]\right|\1_{M_n}\leq \sqrt{K(J)} M_n\|\langle  Q_J h_0,\psi^J\rangle_\mu' (G_b^{-1/2}S)_l^-\|,
\end{align*}
where $M_n=\big\{\max_j|Y_i\widetilde b_j(W_i) - \E[Y\widetilde b_j(W)]|\leq n^{1/6}\big\}$.
\end{proof}

\begin{lemma}\label{lemma:adapt}
Let Assumptions \ref{A:ident}(i)(iii), \ref{A:link}(i), and \ref{A:adapt} be satisfied and consider the mildly ill-posed case, i.e., $\tau_j\sim  j^{a/d}$. Then, we have $\inf_{h_0\in\mathcal H}\Pr_{h_0}(J_0\in\widehat {\mathcal J})\to 1$.
\end{lemma}
\begin{proof}
Let $\mathcal E_n$ denote the event upon which $\overline J(c_1)\leq \widehat J_{\max}\leq \overline J(c_2)$ for some constant $c_1,c_2>0$  and observe that $\Pr(\mathcal E_n^c) = o(1)$ by Lemma \ref{lemma:bounds:hat_J}.
For all $J> J_0$, we make use of the upper bound
\begin{equation*}
\big| \widehat f_{J_0}-\widehat f_J\big|
\leq \big| \widehat f_{J_0}-f(Q_{J_0} h_0)\big|
+\big| \widehat f_J-f(Q_Jh_0)\big|
+2\big| f(Q_{J_0} h_0 - h_0)\big|
+2\big|f(Q_J h_0 - h_0)\big|.
\end{equation*}
By Lemma \ref{lemma:uniform:bound}, uniformly  for all $J\in\mathcal I$ it holds
\begin{align*}
\big| \widehat f_J-f(Q_Jh_0)\big|\leq 2\overline\sigma_Y^2s_J^{-2} \, \frac{\sqrt{J\log \overline J}}{n-1}
+\|\langle  Q_J h_0,\psi^J\rangle_\mu' (G_b^{-1/2}S)_l^-\|/\sqrt n
\end{align*}
on some set $\mathcal E_{n,1}\subseteq \mathcal E_n$ where $\Pr(\mathcal E_{n,1}^c)=o(1)$.
For all $J>J_0$ we have
\begin{align*}
\Big|f(Q_Jh_0 - h_0)\Big|&\leq C \|\Pi_J h_0 - h_0\|_\mu^2\leq CLJ^{-2p/d}\leq CL(J_0^+)^{-2p/d}\\
& \leq C_0L \tau_{J_0^+}^{2}\, \sqrt{J_0^+ \log n}/n,
\end{align*}
where in the last inequality we used the definition of $J_0^+=\inf\{J\in\mathcal J:J>J_0\}$.
Hence, we conclude for all $J\geq J_0^+$ that
\begin{align*}
\big| \widehat f_{J_0}-\widehat f_J\big|
&\leq (C_0+2\overline\sigma_Y^2)\left(s_{J_0}^{-2} \sqrt{J_0\log n}/n +s_J^{-2} \sqrt{J \log n}/n\right)\vee 1/\sqrt n.
\end{align*}
Due to Lemma \ref{lemma:bounds:tau} it holds $s_J^{-2}\leq (1+\eta)^{2}\,\widehat s_J^{-2}$ for some $\eta\in(0,1)$, uniformly for $J\in\mathcal I\cap\{J>J_0\}$,  on some set $\mathcal E_{n,2}$ with $\Pr(\mathcal E_{n,2}^c)=o(1)$. Consequently, on $\mathcal E_{n,1}\cap\mathcal E_{n,2}$ it holds
\begin{align*}
\big| \widehat f_{J_0}-\widehat f_J\big|
&\leq (C_0+2\overline\sigma_Y^2)(1+\eta)^2\left(\widehat s_{J_0}^{-2} \sqrt{J_0\log n}/n +\widehat s_J^{-2} \sqrt{J\log n}/n\right)\vee 1/\sqrt n\\
&= (C_0+2\overline\sigma_Y^2)(1+\eta)^2\left(\widehat V(J_0) + \widehat V(J)\right)
\end{align*}
uniformly for $J\in\mathcal I\cap\{J>J_0\}$. We conclude that $J_0\in\widehat{\mathcal J}$ on $\mathcal E_{n,1}\cap\mathcal E_{n,2}$ and $\Pr(\mathcal E_{n,1}\cap\mathcal E_{n,2})\to 1$.
\end{proof}

\bibliography{BiB}

\begin{thebibliography}{35}
\providecommand{\natexlab}[1]{#1}
\providecommand{\url}[1]{\texttt{#1}}
\expandafter\ifx\csname urlstyle\endcsname\relax
  \providecommand{\doi}[1]{doi: #1}\else
  \providecommand{\doi}{doi: \begingroup \urlstyle{rm}\Url}\fi

\bibitem[Bickel and Ritov(1988)]{bickel1988estimating}
P.~J. Bickel and Y.~Ritov.
\newblock Estimating integrated squared density derivatives: sharp best order
  of convergence estimates.
\newblock \emph{Sankhy{\=a}: The Indian Journal of Statistics, Series A}, pages
  381--393, 1988.

\bibitem[Birg{\'e} and Massart(1995)]{birge1995}
L.~Birg{\'e} and P.~Massart.
\newblock Estimation of integral functionals of a density.
\newblock \emph{The Annals of Statistics}, pages 11--29, 1995.

\bibitem[Blundell et~al.(2007)Blundell, Chen, and
  Kristensen]{BCK07econometrica}
R.~Blundell, X.~Chen, and D.~Kristensen.
\newblock Semi-nonparametric iv estimation of shape-invariant engel curves.
\newblock \emph{Econometrica}, 75\penalty0 (6):\penalty0 1613--1669, 2007.

\bibitem[Breunig and Chen(2021)]{BC2020}
C.~Breunig and X.~Chen.
\newblock Adaptive, rate-optimal hypothesis testing in nonparametric iv models.
\newblock \emph{arXiv preprint arXiv:2006.09587v2}, 2021.

\bibitem[Breunig and Johannes(2016)]{breunig2016}
C.~Breunig and J.~Johannes.
\newblock Adaptive estimation of functionals in nonparametric instrumental
  regression.
\newblock \emph{Econometric Theory}, 32\penalty0 (3):\penalty0 612--654, 2016.

\bibitem[Butucea(2007)]{butucea2007}
C.~Butucea.
\newblock Goodness-of-fit testing and quadratic functional estimation from
  indirect observations.
\newblock \emph{The Annals of Statistics}, 35\penalty0 (5):\penalty0
  1907--1930, 2007.

\bibitem[Butucea and Meziani(2011)]{butucea2011}
C.~Butucea and K.~Meziani.
\newblock Quadratic functional estimation in inverse problems.
\newblock \emph{Statistical Methodology}, 8\penalty0 (1):\penalty0 31--41,
  2011.

\bibitem[Cai and Low(2006)]{CL06}
T.~T. Cai and M.~G. Low.
\newblock Optimal adaptive estimation of a quadratic functional.
\newblock \emph{The Annals of Statistics}, 34\penalty0 (5):\penalty0
  2298--2325, 2006.

\bibitem[Carrasco et~al.(2007)Carrasco, Florens, and Renault]{CFR06handbook}
M.~Carrasco, J.-P. Florens, and E.~Renault.
\newblock Linear inverse problems in structural econometrics: Estimation based
  on spectral decomposition and regularization.
\newblock In J.~Heckman and E.~Leamer, editors, \emph{Handbook of
  Econometrics}, volume~6B. North Holland, 2007.

\bibitem[Chen and Christensen(2018)]{ChenChristensen2017}
X.~Chen and T.~M. Christensen.
\newblock Optimal sup-norm rates and uniform inference on nonlinear functionals
  of nonparametric iv regression.
\newblock \emph{Quantitative Economics}, 9\penalty0 (1):\penalty0 39--84, 2018.

\bibitem[Chen and Pouzo(2015)]{chen2013}
X.~Chen and D.~Pouzo.
\newblock Sieve wald and qlr inferences on semi/nonparametric conditional
  moment models.
\newblock \emph{Econometrica}, 83\penalty0 (3):\penalty0 1013--1079, 2015.

\bibitem[Chen and Rei\ss(2011)]{ChenReiss2011}
X.~Chen and M.~Rei\ss.
\newblock On rate optimality for ill-posed inverse problems in econometrics.
\newblock \emph{Econometric Theory}, 27\penalty0 (03):\penalty0 497--521, 2011.

\bibitem[Chen et~al.(2021)Chen, Christensen, and Kankanala]{chen2021}
X.~Chen, T.~Christensen, and S.~Kankanala.
\newblock Adaptive estimation and uniform confidence bands for nonparametric
  iv.
\newblock \emph{arXiv preprint arXiv:2107.11869}, 2021.

\bibitem[Chernozhukov et~al.(2014)Chernozhukov, Chetverikov, and Kato]{CCK}
V.~Chernozhukov, D.~Chetverikov, and K.~Kato.
\newblock Anti-concentration and honest, adaptive confidence bands.
\newblock \emph{The Annals of Statistics}, 42\penalty0 (5):\penalty0
  1787--1818, 2014.

\bibitem[Chesneau(2011)]{che201}
C.~Chesneau.
\newblock On adaptive wavelet estimation of a quadratic functional from a
  deconvolution problem.
\newblock \emph{Annals of the Institute of Statistical Mathematics},
  63\penalty0 (2):\penalty0 405--429, 2011.

\bibitem[Collier et~al.(2017)Collier, Comminges, and Tsybakov]{collier2017}
O.~Collier, L.~Comminges, and A.~B. Tsybakov.
\newblock Minimax estimation of linear and quadratic functionals on sparsity
  classes.
\newblock \emph{The Annals of Statistics}, 45\penalty0 (3):\penalty0 923--958,
  2017.

\bibitem[Darolles et~al.(2011)Darolles, Fan, Florens, and Renault]{DFFR}
S.~Darolles, Y.~Fan, J.-P. Florens, and E.~Renault.
\newblock Nonparametric instrumental regression.
\newblock \emph{Econometrica}, 79\penalty0 (5):\penalty0 1541--1565, 2011.

\bibitem[Donoho and Nussbaum(1990)]{donoho1990minimax}
D.~L. Donoho and M.~Nussbaum.
\newblock Minimax quadratic estimation of a quadratic functional.
\newblock \emph{Journal of Complexity}, 6\penalty0 (3):\penalty0 290--323,
  1990.

\bibitem[Efromovich and Koltchinskii(2001)]{EfromovichKoltchinskii2001}
S.~Efromovich and V.~Koltchinskii.
\newblock On inverse problems with unknown operators.
\newblock \emph{IEEE Transactions on Information Theory}, 47\penalty0
  (7):\penalty0 2876--2894, 2001.

\bibitem[Efromovich and Low(1996)]{efromovich1996}
S.~Efromovich and M.~Low.
\newblock On optimal adaptive estimation of a quadratic functional.
\newblock \emph{The Annals of Statistics}, 24\penalty0 (3):\penalty0
  1106--1125, 1996.

\bibitem[Fan(1991)]{Fan91}
J.~Fan.
\newblock On the estimation of quadratic functionals.
\newblock \emph{The Annals of Statistics}, pages 1273--1294, 1991.

\bibitem[Gin{\'e} and Nickl(2008)]{gine2008}
E.~Gin{\'e} and R.~Nickl.
\newblock A simple adaptive estimator of the integrated square of a density.
\newblock \emph{Bernoulli}, 14\penalty0 (1):\penalty0 47--61, 2008.

\bibitem[Gin{\'e} and Nickl(2021)]{nicklbook}
E.~Gin{\'e} and R.~Nickl.
\newblock \emph{Mathematical foundations of infinite-dimensional statistical
  models}.
\newblock Cambridge University Press, 2021.

\bibitem[Hall and Horowitz(2005)]{HH}
P.~Hall and J.~L. Horowitz.
\newblock Nonparametric methods for inference in the presence of instrumental
  variables.
\newblock \emph{The Annals of Statistics}, 33:\penalty0 2904--2929, 2005.

\bibitem[Horowitz(2014)]{horowitz2014adaptive}
J.~L. Horowitz.
\newblock Adaptive nonparametric instrumental variables estimation: Empirical
  choice of the regularization parameter.
\newblock \emph{Journal of Econometrics}, 180\penalty0 (2):\penalty0 158--173,
  2014.

\bibitem[Houdr{\'e} and Reynaud-Bouret(2003)]{houdre2003}
C.~Houdr{\'e} and P.~Reynaud-Bouret.
\newblock Exponential inequalities, with constants, for u-statistics of order
  two.
\newblock In \emph{Stochastic inequalities and applications}, pages 55--69.
  Springer, 2003.

\bibitem[Kroll(2019)]{kroll2019rate}
M.~Kroll.
\newblock Rate optimal estimation of quadratic functionals in inverse problems
  with partially unknown operator and application to testing problems.
\newblock \emph{ESAIM: Probability and Statistics}, 23:\penalty0 524--551,
  2019.

\bibitem[Laurent and Massart(2000)]{laurent2000}
B.~Laurent and P.~Massart.
\newblock Adaptive estimation of a quadratic functional by model selection.
\newblock \emph{The Annals of Statistics}, pages 1302--1338, 2000.

\bibitem[Lepski(1990)]{Lepski90}
O.~V. Lepski.
\newblock On a problem of adaptive estimation in gaussian white noise.
\newblock \emph{Theory of Probability and its Applications}, 35:\penalty0
  454–466, 1990.

\bibitem[Lepski and Spokoiny(1997)]{lepski1997}
O.~V. Lepski and V.~G. Spokoiny.
\newblock Optimal pointwise adaptive methods in nonparametric estimation.
\newblock \emph{The Annals of Statistics}, pages 2512--2546, 1997.

\bibitem[Lepski et~al.(1997)Lepski, Mammen, and Spokoiny]{lepski1997optimal}
O.~V. Lepski, E.~Mammen, and V.~G. Spokoiny.
\newblock Optimal spatial adaptation to inhomogeneous smoothness: an approach
  based on kernel estimates with variable bandwidth selectors.
\newblock \emph{The Annals of Statistics}, pages 929--947, 1997.

\bibitem[Liu et~al.(2021)Liu, Mukherjee, Robins, and Tchetgen]{TT2021}
L.~Liu, R.~Mukherjee, J.~M. Robins, and E.~T. Tchetgen.
\newblock Adaptive estimation of nonparametric functionals.
\newblock \emph{Journal of Machine Learning Research}, 22\penalty0
  (99):\penalty0 1--66, 2021.

\bibitem[Newey and Powell(2003)]{NP03}
W.~K. Newey and J.~L. Powell.
\newblock Instrumental variable estimation of nonparametric models.
\newblock \emph{Econometrica}, 71:\penalty0 1565--1578, 2003.

\bibitem[Spokoiny and Vial(2009)]{spokoiny2009parameter}
V.~Spokoiny and C.~Vial.
\newblock Parameter tuning in pointwise adaptation using a propagation
  approach.
\newblock \emph{The Annals of Statistics}, 37\penalty0 (5B):\penalty0
  2783--2807, 2009.

\bibitem[Spokoiny and Willrich(2019)]{spokoiny2019bootstrap}
V.~Spokoiny and N.~Willrich.
\newblock Bootstrap tuning in gaussian ordered model selection.
\newblock \emph{The Annals of Statistics}, 47\penalty0 (3):\penalty0
  1351--1380, 2019.

\end{thebibliography}

\end{document}